\documentclass[a4paper,12pt]{amsart}
\usepackage{amssymb}
\usepackage{ifthen}
\usepackage{cite}
\usepackage[dvips]{graphicx}
\nonstopmode \numberwithin{equation}{section}
\usepackage{amssymb}
\usepackage{ifthen}
\usepackage{graphicx}
\usepackage{amsmath}
\usepackage[T1]{fontenc} 
\usepackage[utf8]{inputenc}
\usepackage[usenames,dvipsnames]{color}
\usepackage{color}
\usepackage[english]{babel}
\usepackage{fancyhdr}
\usepackage{fancybox}
\usepackage{tikz}

\nonstopmode \numberwithin{equation}{section}
\theoremstyle{plain}
\newtheorem*{theoA}{Theorem A}
\newtheorem*{theoB}{Theorem B}

\newtheorem*{lemA}{Lemma A}
\newtheorem*{lemB}{Lemma B}
\newtheorem*{lemC}{Lemma C}
\newtheorem*{lemD}{Lemma D}

\newtheorem{defi}{Definition}[section]

\newtheorem{conj}{Conjecture}

\theoremstyle{definition}
\newtheorem{defn}{Definition}[section]

\newtheorem{thm}{Theorem}[section]
\newtheorem{prob}{Problem}[section]
\newtheorem{cor}{Corollary}[section]
\newtheorem{ques}{Question}[section]
\newtheorem{prop}{Proposition}[section]
\newtheorem{rem}{Remark}[section]
\newtheorem{lem}{Lemma}[section]

\newcounter{minutes}
\divide\time by 60
\newcounter{hours}
\multiply\time by 60
\addtocounter{minutes}{-\time}

\newcounter {own}
\def\theown {\thesection       .\arabic{own}}

\newenvironment{pf}[1][]{%
 \vskip 3mm
 \noindent
 \ifthenelse{\equal{#1}{}}%
  {{\slshape Proof. }}%
  {{\slshape #1.} }%
 }%
{\qed\bigskip}

\newcounter{alphabet}

\def\be{\begin{equation}}
\def\ee{\end{equation}}

\newcommand{\bee}{\begin{enumerate}}
\newcommand{\eee}{\end{enumerate}}

\newcommand{\blem}{\begin{lem}}
\newcommand{\elem}{\end{lem}}
\newcommand{\bthm}{\begin{thm}}
\newcommand{\ethm}{\end{thm}}
\newcommand{\bcor}{\begin{cor}}
\newcommand{\ecor}{\end{cor}}
\newcommand{\beg}{\begin{examp}}
\newcommand{\eeg}{\end{examp}}
\newcommand{\begs}{\begin{examples}}
\newcommand{\eegs}{\end{examples}}

\newcommand{\bdefn}{\begin{defn}}
\newcommand{\edefn}{\end{defn}}

\newcommand{\bprob}{\begin{prob}}
\newcommand{\eprob}{\end{prob}}
\newcommand{\bei}{\begin{itemize}}
\newcommand{\eei}{\end{itemize}}

\newcommand{\bcon}{\begin{conj}}
\newcommand{\econ}{\end{conj}}
\newcommand{\bcons}{\begin{conjs}}
\newcommand{\econs}{\end{conjs}}
\newcommand{\bprop}{\begin{prop}}
\newcommand{\eprop}{\end{prop}}
\newcommand{\br}{\begin{rem}}
\newcommand{\er}{\end{rem}}
\newcommand{\brs}{\begin{rems}}
\newcommand{\ers}{\end{rems}}
\newcommand{\bo}{\begin{obser}}
\newcommand{\eo}{\end{obser}}
\newcommand{\bos}{\begin{obsers}}
\newcommand{\eos}{\end{obsers}}
\newcommand{\bpf}{\begin{pf}}
\newcommand{\epf}{\end{pf}}
\newcommand{\ba}{\begin{array}}
\newcommand{\ea}{\end{array}}
\newcommand{\beq}{\begin{eqnarray}}
\newcommand{\beqq}{\begin{eqnarray*}}
\newcommand{\eeq}{\end{eqnarray}}
\newcommand{\eeqq}{\end{eqnarray*}}

\begin{document}
\title{The Bohr Phenomenon  for Close-to-Convex Harmonic Mappings}

\author{Molla Basir Ahamed$^*$}
\address{Molla Basir Ahamed, Department of Mathematics, Jadavpur University, Kolkata-700032, West Bengal, India.}
\email{mbahamed.math@jadavpuruniversity.in}

\author{Partha Pratim Roy}
\address{Partha Pratim Roy, Department of Mathematics, Jadavpur University, Kolkata-700032, West Bengal,India.}
\email{pproy.math.rs@jadavpuruniversity.in}

\subjclass[{AMS} Subject Classification:]{Primary  30C62, 30C80, Secondary 30C45}

\def\thefootnote{}
\footnotetext{ {\tiny File:~\jobname.tex,
printed: \number\year-\number\month-\number\day,
          \thehours.\ifnum\theminutes<10{0}\fi\theminutes }
} \makeatletter\def\thefootnote{\@arabic\c@footnote}\makeatother

\begin{abstract} 
The classical Bohr inequality states that if $f(z)=\sum_{n=0}^{\infty} a_n z^n$ is analytic and $|f(z)|<1$ in the unit disk $\mathbb{D}$, then $\sum_{n=0}^{\infty} |a_n| r^n \le 1$ for $|z|=r \le 1/3$, where $1/3$ is sharp. Extending this to harmonic mappings $f=h+\overline{g}$ is central in geometric function theory due to the co-analytic part $g$. This paper establishes sharp Bohr-type inequalities for two classes of sense-preserving close-to-convex harmonic mappings. Let $\mathcal{H}_0$ be the class of harmonic mappings $f=h+\overline{g}$ in $\mathbb{D}$ normalized by $h(0)=g(0)=h'(0)-1=g'(0)=0$. We introduce:
\[ \mathcal{P}_{\mathcal{H}_0}(M) := \{ f \in \mathcal{H}_0 : \text{Re}(zh''(z)) > -M + |zg''(z)|, \; z \in \mathbb{D}, \; M > 0 \} \]
\[ \mathcal{W}_{\mathcal{H}_0}(\alpha,\beta) := \{ f \in \mathcal{H}_0 : \text{Re}(h'(z) + \alpha zh''(z) - \beta) > |g'(z) + \alpha zg''(z)|, \; z \in \mathbb{D} \} \]
where $\alpha \ge 0$, $\beta < 1$.

We prove generalized Bohr inequalities by replacing the basis $\{r^n\}$ with non-negative continuous functions $\{\varphi_n(r)\}$. The results are proved using sharp coefficient bounds and growth theorems, providing new insights into the Bohr phenomenon for harmonic mappings and subclasses defined by differential inequalities. 
\end{abstract}

\maketitle
\pagestyle{myheadings}
\markboth{M. B. Ahamed, P. P. Roy}{The Bohr Phenomenon  for Close-to-Convex Harmonic Mappings}
\section{\bf Introduction}
Let $\mathbb{D} := \{z \in \mathbb{C} : |z| < 1\}$ denote the open unit disk in the complex plane. Let $\mathcal{B}$ denote the class of analytic functions $f$ on $\mathbb{D}$ satisfying $|f(z)| < 1$ for all $z \in \mathbb{D}$. In 1914, Bohr \cite{Bohr-1914} established a classical inequality for the class $\mathcal{B}$, which has since initiated a major line of research in geometric function theory.
\begin{theoA}\cite{Bohr-1914}
	If $ f(z)=\sum_{n=0}^{\infty}a_nz^n\in\mathcal{B} $, then 
	\begin{equation}\label{e-1.1}
		M_f(r):=\sum_{n=0}^{\infty}|a_n|r^n\leq 1 \;\; \mbox{for}\;\; |z|=r\leq\frac{1}{3},
	\end{equation}
	where $a_k=f^{(k)}(0)/k!$ for $k\geq 0$. The constant $1/3$ is best possible. 
\end{theoA}
The constant $1/3$ is known as the \textit{Bohr radius} and the inequality in \eqref{e-1.1} is called the \textit{Bohr inequality} for the class $\mathcal{B}$.\vspace{1.2mm}

Following Dixon's seminal formulation \cite{Dixon & BLMS & 1995}, which established a fundamental link between the classical Bohr inequality and Banach algebras satisfying the von Neumann inequality, the Bohr phenomenon has undergone extensive generalization for various function classes and abstract functional settings (see, e.g., \cite{Ponnusmy-Survey}). Recent investigations in geometric function theory have increasingly focused on the critical role played by the initial term $|f(0)|$ within the majorant series $M_f(r)$, demonstrating that its modification fundamentally alters the resulting Bohr radius. For instance, substituting the initial term $|f(0)|$ with the power variant $|f(0)|^p$ for $0 < p \le 2$ yields a sharp generalized Bohr radius of $p/(p+2)$ \cite{Liu-Ponnusamy-PAMS-2021}. Furthermore, Alkhaleefah \emph{et al.} \cite{Alkhaleefah-Kayumov-Ponnusamy-PAMS-2019} proved that replacing the static initial value $|f(0)|$ with the local functional value $|f(z)|$ shifts the classical radius $1/3$ to the sharp lower bound $\sqrt{5}-2$.\vspace{1.2mm}

 The generalization of Bohr’s theorem remains a highly active area of modern analysis. Over the past few decades, investigations on the Bohr inequality have expanded for various functional and geometric settings. For instance, the Bohr property has been established for bases of holomorphic functions \cite{aizenberg-2001, Aytuna-Djakov-BLMS-2013}, extended to abstract Banach algebras \cite{Paulsen-PLMS-2002}, and evaluated for holomorphic mappings taking values in homogeneous balls \cite{Hamada-IJM-2009}. In the setting of several complex variables, Galicer \textit{et al.} \cite{Galicer-Mansilla-Muro-TAMS-2020} introduced and analyzed the mixed Bohr radius.\vspace{1.2mm}
 
 A broader account of recent advancements and multi-author variations of these inequalities can be found throughout the literature (see, e.g., \cite{Ahamed-CMFT-2022, Ahamed-Roy-PMS-2025, Ahamed-Allu-MS-2023, Ahamed-Ahammed-Hamada-MAS-2026, Ahamed-Allu-Halder-CRM-2025, Ahamed-Rudrani-Sharma-JCMA-2025, Ahamed-Allu-Halder-BSM-2025, Ahammed-Ahamed-Roy-Filomat-2026, Ahamed-CVEE-2024, Ahamed-Ahammed-MJM-2024, Ahamed-Allu-BMMSS-2022, Ahamed-Allu-CMB-2023, Ahamed-Allu-Halder-CVEE-2023, Aha-Aha-CMFT-2023, Aha-Allu-RMJ-2022, Alkhaleefah-Kayumov-Ponnusamy-PAMS-2019, Allu-Arora-JMAA-2022, Allu-CMB-2022, Das-JMAA-2022, Lata-Singh-PAMS-2022, S. Kumar-PAMS-2022, Kumar-JMAA-2023,Ali-Abdul-NG-CVEE-2016,Ponn-Starkov-JMAA-2024} and the references therein).\vspace{1.2mm} 
 
 The study of the Bohr inequality within the framework of complex-valued harmonic mappings has emerged as a deeply compelling and rapidly expanding area of geometric function theory over the past decade \cite{Abu-Muhanna-CVEQ-2010, Abu-Muhanna-JMAA-2014, Allu-Halder-Indag-2022, Ahamed-Allu-Halder-AMP-2021}. A complex-valued harmonic mapping $f$ defined in the open unit disk $\mathbb{D} $ admits a unique canonical representation $f = h + \overline{g}$, where $h$ and $g$ are analytic functions known respectively as the analytic and co-analytic parts of $f$ \cite{Allu-Halder-Indag-2022, Ahamed-Allu-Halder-AMP-2021}. Unlike purely analytic functions, harmonic mappings require distinct handling due to the presence of the co-analytic component, making the determination of their classical and refined Bohr radii a challenging open problem for many subclasses \cite{Allu-Halder-Indag-2022}.\vspace{1.2mm}
 
 The foundational investigation into this domain was spearheaded by Abu-Muhanna \cite{Abu-Muhanna-CVEQ-2010}, who first introduced Bohr's phenomenon to bounded harmonic classes and subordination frames. This framework was later extended significantly by Abu-Muhanna\emph{et al.} \cite{Abu-Muhanna-JMAA-2014}, who explored the sharp Bohr radius for various subordinating families of bounded harmonic mappings.\vspace{1.2mm}
 
 More recently, research has shifted toward establishing specialized Bohr-type inequalities for structurally defined geometric subclasses of harmonic mappings \cite{Allu-Halder-Indag-2022, Ahamed-Allu-Halder-AMP-2021}. For instance, Allu and Halder \cite{Allu-Halder-Indag-2022} evaluated the Bohr phenomenon for sense-preserving harmonic mappings whose analytic parts belong to Ma-Minda type convex classes or convex families defined with respect to conjugate points. Concurrently, Ahamed \emph{et al.} \cite{Ahamed-Allu-Halder-AMP-2021} successfully determined the sharp Bohr-Rogosinski radius, refined Bohr radius, and improved Bohr radius versions incorporating area measures for specific close-to-convex harmonic mappings. These ongoing advancements continue to provide deep insights into how the geometric properties of the image domains influence the coefficient bounds and majorant series of harmonic functions \cite{Allu-Halder-Indag-2022, Ahamed-Allu-Halder-AMP-2021}.
 
\begin{defi}\cite{Grigoryan-POTENTIAL-2023}
	Let $A\subset\hat{\mathbb{C}}$ be a set, a function $f: A\rightarrow\mathbb{C}$ is said to be harmonic in an open set $\Omega\subset A,$ provided $f$ is continuous in $\Omega$ and $f$ is twice continuously differentiable in $\Omega\setminus\{\infty\}$ and satisfies the Laplace equation 
	\begin{align*}
		\dfrac{\partial^2{f(z)}}{\partial x^2}+\dfrac{\partial^2{f(z)}}{\partial y^2}=0,\; z=x+iy\in \Omega\setminus\{\infty\}.
	\end{align*}
\end{defi}
A complex-valued function $f = u + iv$ defined on the open unit disk $\mathbb{D} := \{z \in \mathbb{C} : |z| < 1\}$ is called a harmonic mapping if its real and imaginary parts, $u$ and $v$, are real harmonic functions in $\mathbb{D}$. Such a function satisfies the Laplace equation $\Delta f = 4f_{z\bar{z}} = 0$, which ensures that $f$ admits the canonical representation $f = h + \overline{g}$, where $h$ and $g$ are analytic in $\mathbb{D}$. The Jacobian of $f$ is given by $J_f(z) = |h'(z)|^2 - |g'(z)|^2$. According to Lewy's theorem \cite{Lewy-BAMS-1936}, $f$ is locally univalent and sense-preserving in $\mathbb{D}$ if and only if $J_f(z) > 0$ for all $z \in \mathbb{D}$. This condition is equivalent to requiring that $h'(z) \neq 0$ and that the second complex dilation $\omega_f(z) := g'(z)/h'(z)$ satisfies $|\omega_f(z)| < 1$ for all $z \in \mathbb{D}$.\vspace{1.2mm}

Let $\mathcal{H}(\mathbb{D})$ denote the class of all complex-valued harmonic functions $f=h+\overline{g}$ in $\mathbb{D}$ normalized by $h(0)=h'(0)-1=0$ and $g(0)=0$. A function $f\in\mathcal{H}(\mathbb{D})$ belongs to the subclass $\mathcal{H}_{0}$ if $g'(0)=0$. Consequently, every $f=h+\overline{g}\in\mathcal{H}_{0}$ can be expressed in the form
\begin{align}\label{Eq-1.2}
	f(z)=h(z)+\overline{g(z)}
	=z+\sum_{n=2}^{\infty}a_n z^n
	+\overline{\sum_{n=2}^{\infty}b_n z^n}.
\end{align}
Harmonic mappings serve as a powerful analytical tool across various branches of science and engineering, particularly in the study of fluid dynamics \cite{Aleman-2012, Constantin-2017}. Notably, Aleman and Constantin \cite{Aleman-2012} established a profound connection between univalent harmonic mappings and ideal fluid flows. This framework introduced an ingenious technique for solving the incompressible, two-dimensional Euler equations, demonstrating the deep geometric interplay between complex analysis and fluid mechanics (see \cite{Constantin-2017} for a comprehensive review).\vspace{1.2mm}
\subsection{Recent development of Bohr inequality for analytic functions}

In \cite{Kayumov-Ponnusamy-Shakirov-MN-2017}, Kayumov \textit{et al.} presented the Bohr radius for locally univalent planar harmonic mappings. As part of the recent developments in this direction, several improved versions of the Bohr inequality for harmonic mappings were discussed by Evdoridis \textit{et al.} in \cite{Evdoridis-Ponnusamy-Rasila-IM-2019}. Various improved forms of the classical Bohr inequality were investigated by Kayumov and Ponnusamy in \cite{Kayumov-Ponnusamy-CRMath-2018, Kayumov-Ponnusamy-JMAA-2018}. Recently, a generalized form of the Bohr sum was studied by Kayumov \textit{et al.} \cite{Kayumov-Khammatova-Ponnusamy-MJM-2022}, which is described as follows: let $\{\varphi_k(r)\}_{k=0}^{\infty}$ be a sequence of non-negative continuous functions on $[0, 1)$ such that the series 
\[
\varphi_0(r) + \sum_{k=1}^{\infty} \varphi_k(r)
\]
converges locally uniformly for $r \in [0, 1)$.
\begin{theoB}\cite{Kayumov-Khammatova-Ponnusamy-MJM-2022}
	Let $f\in\mathcal{B}, f(z)=\sum_{n=0}^{\infty}a_n z^n,$
	and $p\in(0,2]$. If
	\[
	\varphi_0(r)>
	\frac{2}{p}\sum_{n=1}^{\infty}\varphi_n(r),
	\qquad r\in[0,R),
	\]
	where $R$ is the minimal root of the equation $\varphi_0(r)={2}/{p}\sum_{n=1}^{\infty}\varphi_n(r),$
	then the following sharp inequality holds:
	\[
	A_f(\varphi,p,r)
	=
	|a_0|^p\varphi_0(r)
	+\sum_{n=1}^{\infty}|a_n|\varphi_n(r)
	\leq \varphi_0(r),
	\qquad \text{for all } r\leq R.
	\]
	In the case when 
	\begin{align*}
		\varphi_0(r)
		<
		\frac{2}{p}\sum_{n=1}^{\infty}\varphi_n(r)
	\end{align*}
	in some interval $(R,R+\varepsilon)$, the number $R$ cannot be improved. If the functions $\varphi_n(r)$, $n\geq 0$, are smooth, then the last condition is equivalent to the inequality
	\begin{align*}
		\varphi_0'(R)
		<
		\frac{2}{p}\sum_{n=1}^{\infty}\varphi_n'(R).
	\end{align*}
\end{theoB}
A more general form of Theorem~B and its applications to the computation of Bohr-type radii associated with various well-known integral operators were presented in \cite{Kumar-2023}. A careful examination of Theorem~B and the techniques employed in its proof suggests that the Bohr phenomenon extends naturally to certain subclasses of harmonic mappings. In preparation for our main results, we begin by recalling the notion of subordination for harmonic mappings.\vspace{1.2mm}

Since, harmonic mappings are natural extension of analytic functions on a simply connected domain, therefore it is natural to raise the following question.
\begin{ques}\label{Q-1}
		Can we establish a sharp Bohr inequality for certain classes of close-to-convex harmonic mappings by replacing the classical basis $\{r^n\}$ of the majorant series with a sequence $\{\varphi_n(r)\}$ of non-negative continuous functions?
\end{ques}
The paper is organized as follows. In Section~2, we present several 
preliminary lemmas, including sharp coefficient estimates and growth 
theorems for the classes $\mathcal{P}_{\mathcal{H}_0}(M)$ and 
$\mathcal{W}_{\mathcal{H}_0}(\alpha,\beta)$, which serve as essential tools for our main results. Section~3 is devoted to the formulation and proof of the generalized Bohr-type inequalities for the class 
$\mathcal{P}_{\mathcal{H}_0}(M)$, wherein the sharp Bohr radius is explicitly determined and the extremal function is provided. Finally, in Section~4, we establish sharp Bohr-type inequalities for the class 
$\mathcal{W}_{\mathcal{H}_0}(\alpha,\beta)$ and identify its corresponding extremal mappings.
\section{\bf {Bohr inequality and its generalization with new basis for the class $\mathcal{P}^{0}_{\mathcal{H}}(M)$}}

In 2013, Ponnusamy \textit{et al.} \cite{Ponnusamy-CVEQ-2013}
considered the following classes
\[
\mathcal{P}_{\mathcal{H}}
=
\left\{
f=h+\overline{g}\in\mathcal{H}
:
\operatorname{Re} h'(z)>|g'(z)|
\ \text{for}\ z\in\mathbb{D}
\right\}
\]
and $\mathcal{P}_{H}^{0}
=
\mathcal{P}_{H}\cap\mathcal{H}_{0}.$ Motivated by the above classes, Ghosh and Allu \cite{Ghosh2020} have studied the
following classes.
\begin{align*}
	\mathcal{P}_{H}^{0}(M)
	=
	\left\{
	f=h+\overline{g}\in\mathcal{H}_{0}
	:
	\operatorname{Re}\!\bigl(zh''(z)\bigr)
	>
	-M+\left|zg''(z)\right|,\; z\in\mathbb{D}
	\ \text{and}\ M>0
	\right\}.
\end{align*}
To investigate the Bohr inequality and the corresponding Bohr radius for the class $\mathcal{P}_{\mathcal H}^{0}(M)$, it is necessary to establish the coefficient bounds and growth estimates for functions within this class. The following result provides these essential estimates.
\begin{lemA}\cite{Ghosh2020}.
	Let $f=h+\overline{g}\in\mathcal{P}_{\mathcal H}^{0}(M)$ for some $M>0$ and be of the form
	\eqref{Eq-1.2}. Then for $n\ge2$,
	\begin{enumerate}
		\item[(i)] $\displaystyle |a_n| + |b_n|\leq \frac{2M}{n(n-1)}; $\\[1mm]
		
		\item[(ii)] $\displaystyle ||a_n| - |b_n||\leq \frac{2M}{n(n-1)};$\\[1mm]
		
		\item[(iii)] $\displaystyle |a_n|\leq \frac{2M}{n(n-1)}.$
	\end{enumerate}
	
	The inequalities are sharp with extremal function $f$ given by
	\[
	f_M'(z)=1-2M\ln(1-z).
	\]
\end{lemA}

\begin{lemB}\cite{Ghosh2020}
	Let $f\in\mathcal{P}_{\mathcal H}^{0}(M)$. Then
	\begin{align}
		|z|
		+
		2M\sum_{n=2}^{\infty}
		\frac{(-1)^{\,n-1}|z|^n}{n(n-1)}
		\le
		|f(z)|
		\le
		|z|
		+
		2M\sum_{n=2}^{\infty}
		\frac{|z|^n}{n(n-1)}.
	\end{align}
	Both inequalities are sharp for the function $f_M$ given by
	\[
	f_M(z)
	=
	z
	+
	2M\sum_{n=2}^{\infty}
	\frac{z^n}{n(n-1)}.
	\]
	
\end{lemB}
\begin{figure}
	\begin{center}
		\includegraphics[width=0.32\linewidth]{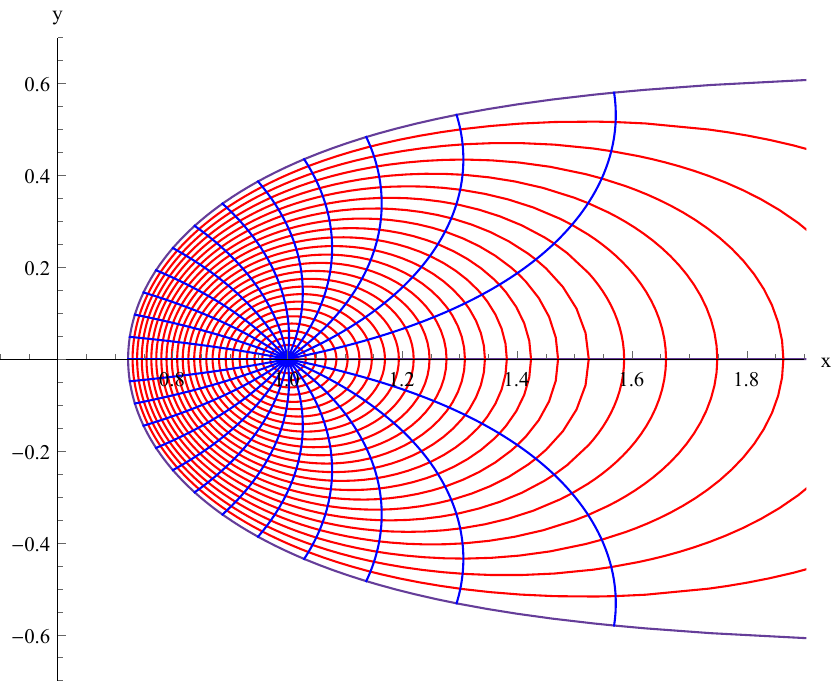}\;\;\;\;\includegraphics[width=0.32\linewidth]{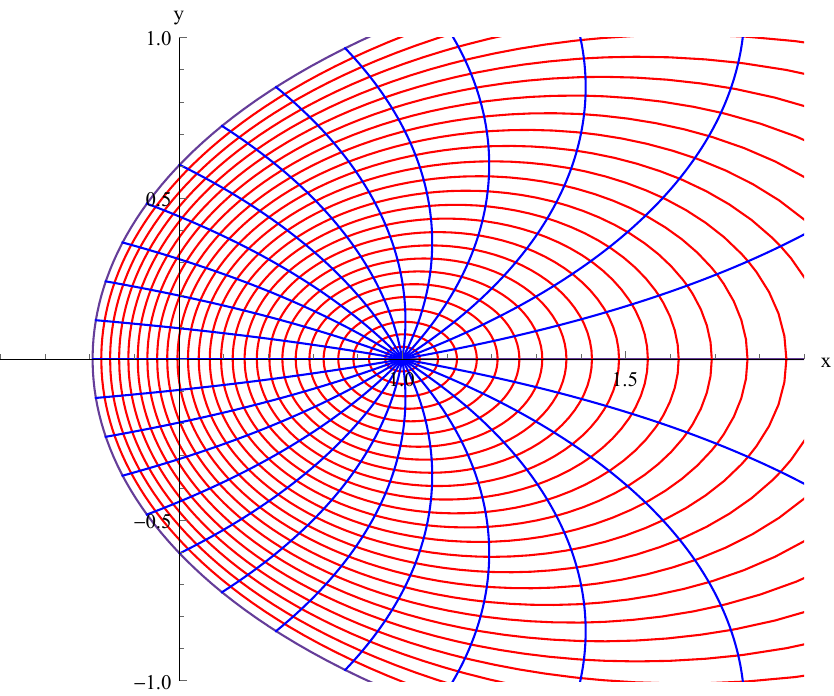}\;\;\;\;\includegraphics[width=0.32\linewidth]{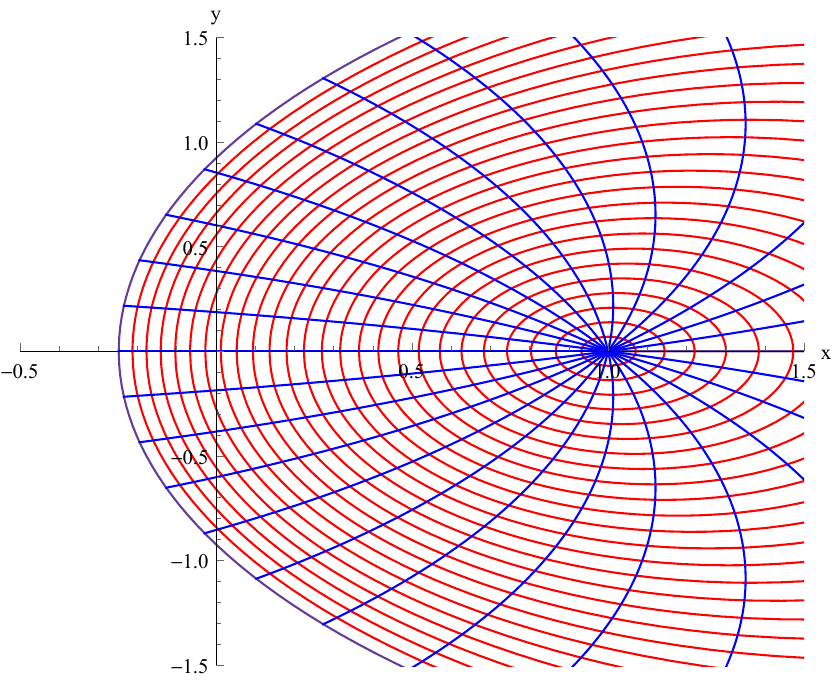}
	\end{center}
	\caption{The figure shows the images of $ f_M^{\prime}(z)=1-2M\, \ln\, (1-z)\in \mathcal{P}_{\mathcal{H}}^{0}(M) $  from left to right, for values of $ M=0.2, 0.5, 0.8 $, respectively.}
\end{figure}

In this section, we present our first main result, which establishes the generalized Bohr inequality for the class $\mathcal{P}^{0}_{\mathcal{H}}(M)$. By applying the coefficient bounds introduced in the previous section, we obtain the following sharp estimate.
\begin{thm}\label{Thm-2.1}
	Let $f \in \mathcal{P}^{0}_{\mathcal{H}}(M)$ be given by \eqref{Eq-1.2} with
	$0<M<{1}/{2(\log 4-1)}$, and let $\{\varphi_n(r)\}$ be a sequence of
	nonnegative increasing functions differentiable in $(0,1)$ such that the
	series $\sum_{n=0}^{\infty}\varphi_n(r)
	$ converges locally uniformly with respect to $r\in[0,1)$, and
	\begin{align}\label{Eq-2.3}
		\sum_{n=2}^{\infty}\frac{\varphi_n(0)}{n(n-1)}
		<
		\frac{1}{2M}+1-2\ln 2.
	\end{align}
	Then the inequality
	\begin{align}\label{Eqnn-2.2}
		B(f,\varphi,r)
		=
		r\varphi_0(r)
		+
		\sum_{n=2}^{\infty}
		\bigl(|a_n|+|b_n|\bigr)\varphi_n(r)
		\leq d\bigl(f(0),\partial f(\mathbb{D})\bigr)
	\end{align}
	holds for all $|z|=r\leq r_f(M)$, where $r_f(M)$ is the unique root
	of the equation
	\begin{align*}
		r\varphi_0(r)
		+
		2M
		\left(
		2\ln 2-1
		+
		\sum_{n=2}^{\infty}\frac{\varphi_n(r)}{n(n-1)}
		\right)
		-1
		=0
	\end{align*}
	in $(0,1)$. The radius $r_f(M)$ is best possible.
\end{thm}
Specifically, by choosing $\varphi_n(r)=r^n$ for all $(n\geq 0)$, Theorem~\ref{Thm-2.1} immediately yields the following corollary. It is worth noting that, under this particular choice of the weight functions, our general result reduces to the classical setting and consequently recovers a previously established theorem as a special case. Thus, Theorem~\ref{Thm-2.1} may be regarded as a genuine extension of the known result.

\begin{cor}\cite[Theorem 2.9]{Allu-Halder-2021}
		Let \(f\in \mathcal{P}_{\mathcal H}^{0}(M)\) be given by \eqref{Eq-1.2}  with
		$0<M<{1}/{2(\log 4-1)}$. Then the inequality 
		\begin{align*}
			r
			+
			\sum_{n=2}^{\infty}
			\bigl(|a_n|+|b_n|\bigr)r^n	\leq d\bigl(f(0),\partial f(\mathbb{D})\bigr)
		\end{align*} 
		holds for \(|z|=r\le r_1(M)\), where \(r_1(M)\) is the unique root of
	\begin{align}\label{Eqn-2.4}
		r
		+
		2M(r-(1-r)\ln(1-r))
		=1+2M(1-2\ln 2).
	\end{align}
		in \((0,1)\). The radius \(r_1(M)\) is best possible.
\end{cor}
\begin{cor}
	Let \(f\in \mathcal{P}_{\mathcal H}^{0}(M)\) be given by \eqref{Eq-1.2}  with
	$0<M<{1}/{2(\log 4-1)}$. Let $\varphi_0(r)=1
	\text{and}
	\varphi_n(r)=(n+1)r^n,\quad n\geq 1.$ Then the inequality 
	\begin{align*}
		r
		+
		\sum_{n=2}^{\infty}
		\bigl(|a_n|+|b_n|\bigr)(n+1)r^n	\leq d\bigl(f(0),\partial f(\mathbb{D})\bigr)
	\end{align*} holds for \(|z|=r\le r_2(M)\), where \(r_2(M)\) is the unique root of the equation
	\begin{align}\label{Eqn-2.5}
		r
		+
		2M(r+(1-2r)\ln(1-r))
		=1+2M(1-2\ln 2).
	\end{align}
	in \((0,1)\). The radius \(r_2(M)\) is best possible.
\end{cor}
\begin{figure}
	\begin{center}
		\includegraphics[width=0.45\linewidth]{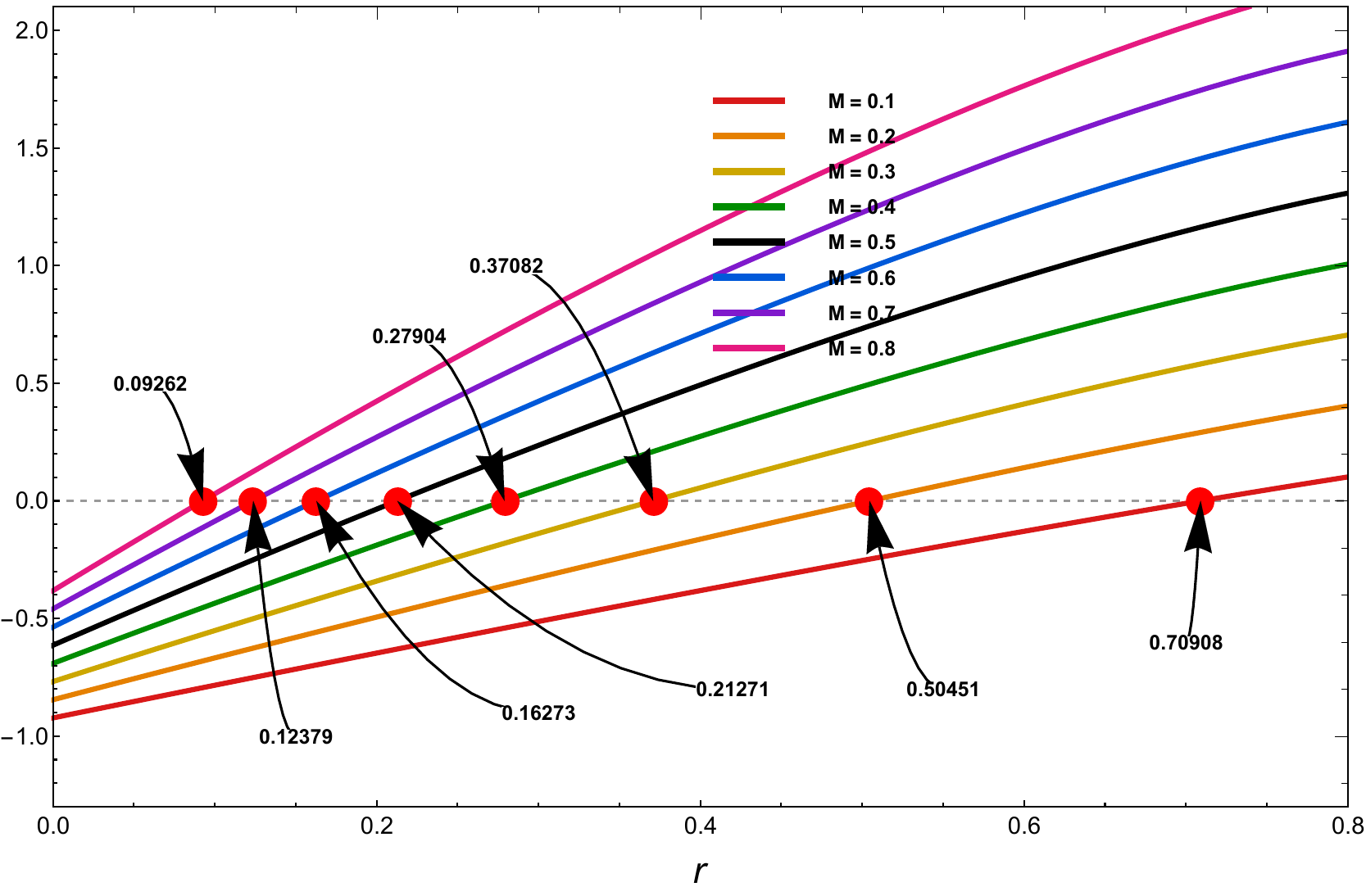}\;\;\;\;\;\;\includegraphics[width=0.45\linewidth]{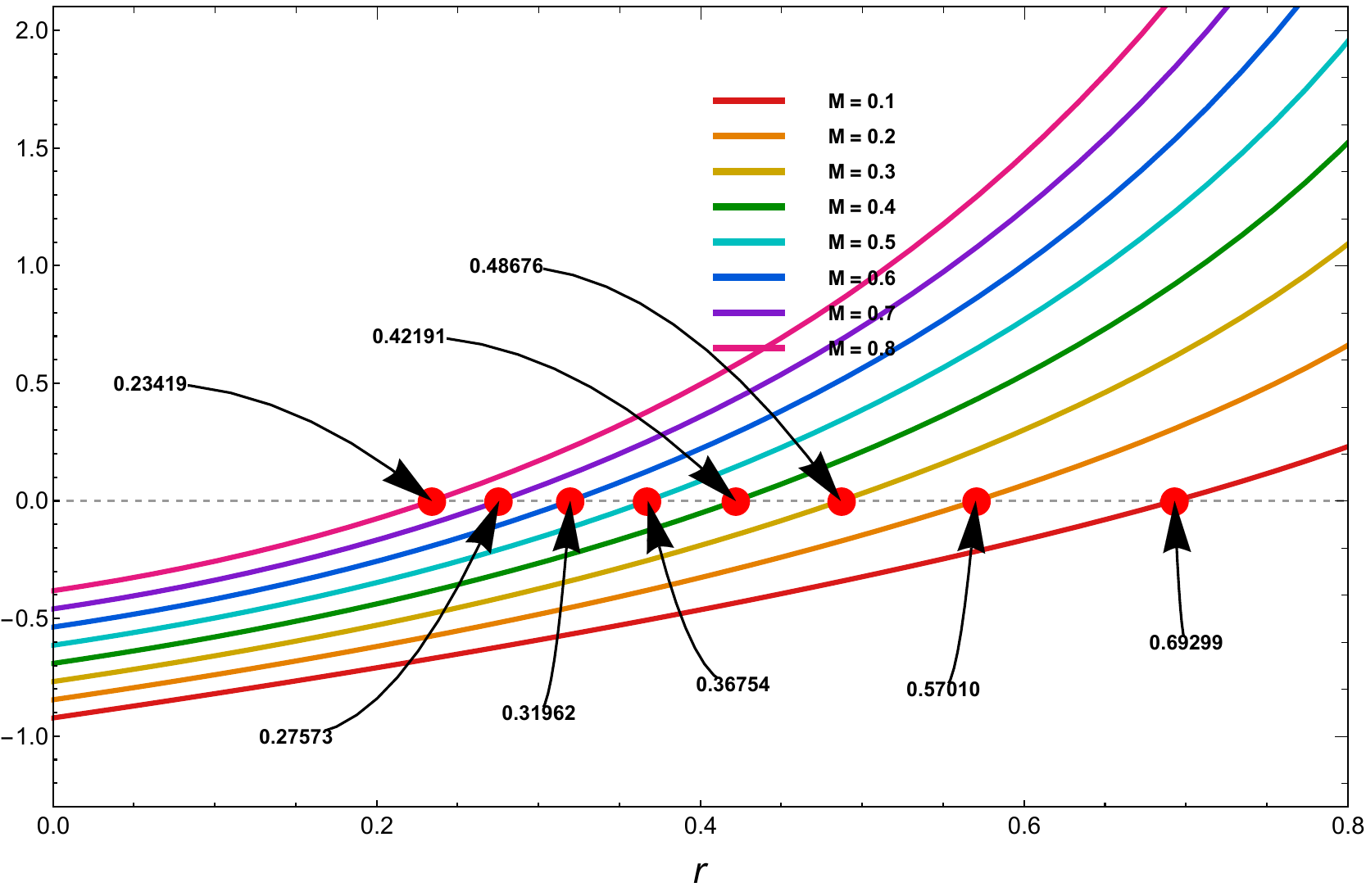}
	\end{center}
	\caption{The figure exhibits the roots of equation \eqref{Eqn-2.4}  and the roots of equation \eqref{Eqn-2.5}  as defined in Table 1 .}
\end{figure}
\begin{table}[ht]
	\centering
	\begin{tabular}{|l|l|l|l|l|l|l|l|l|}
		\hline
		$M$& $0.1 $&$0.2 $& $0.3 $& $0.4$& $0.5 $&$0.6 $&$ 0.7 $ &$0.8$ \\
		\hline
		$r_1(M)$& $0.7090$&$0.5045 $& $0.3708$& $0.2790$& $0.2127$&$0.1627$& $0.1237$& $0.0926$ \\
		\hline
		$ r_2(M)$& $0.6929$&$0.5701 $& $0.4867$& $0.4219$& $0.3675 $&$0.3196 $& $0.2757$& $0.2341$ \\
		\hline
	\end{tabular}\vspace{2.5mm}
	\caption{The table shows the roots of equations \eqref{Eqn-2.4} and \eqref{Eqn-2.5} for various values of $M$.}
	
\end{table}
\begin{proof}[\bf Proof of the Theorem \ref{Thm-2.1}]
	Let $f\in \mathcal{P}_{\mathcal{H}}^{0}(M)$. Then, in view of Lemma B, we have
	\begin{align}\label{Eq-2.2}
		|f(z)|
		\geq
			|z|
		+
		2M\sum_{n=2}^{\infty}
		\frac{(-1)^{\,n-1}|z|^n}{n(n-1)},
		\qquad |z|<1.
	\end{align}
	
From \eqref{Eq-2.2}, we obtain
	\begin{align}\label{Eqn-2.3}
		\liminf_{|z|\to 1}|f(z)|
		\geq
	   1
		+
		2M\sum_{n=2}^{\infty}
		\frac{(-1)^{\,n-1}}{n(n-1)}=1+2M(1-2\ln 2).
	\end{align}
	
	The Euclidean distance between $f(0)$ and the boundary of $f(\mathbb{D})$ is given by
	\begin{align}\label{Eq-2.4}
		d\bigl(f(0),\partial f(\mathbb{D})\bigr)
		=
		\liminf_{|z|\to 1}
		|f(z)-f(0)|.
	\end{align}
	
	Since $f(0)=0$, from \eqref{Eqn-2.3} and \eqref{Eq-2.4}, we obtain
	\begin{align}\label{Eq-2.5}
		d\bigl(f(0),\partial f(\mathbb{D})\bigr)
		\geq
		1+2M(1-2\ln 2).
		\end{align}
		Let $	L_1:[0,1)\to\mathbb{R}$
		be defined by
		\begin{align*}
			L_1(r)
			=
			r\varphi_0(r)
			+
			2M\sum_{n=2}^{\infty}\frac{\varphi_n(r)}{n(n-1)}
			-1
			-
			2M(1-2\ln 2).
		\end{align*}
	In view of \eqref{Eq-2.3}, it is easy to see that
		\begin{align*}
			L_1(0)
			=	2M\sum_{n=2}^{\infty}\frac{\varphi_n(0)}{n(n-1)}
			-1
			-2M(1-2\ln 2)<0
		\end{align*}
		for  $0<M<{1}/{2(\log 4-1)}$. On the other hand, since $	L_1(r)\to +\infty
		\quad \text{as } r\to 1^{-},$
		and
		\begin{align*}
			L_1'(r)
			=
			r\varphi_0'(r)+\varphi_0(r)
			+
			2M
			\sum_{n=2}^{\infty}
			\frac{\varphi_n'(r)}{n(n-1)}
			>0
		\end{align*}
		for $r\in(0,1)$, it follows that $L_1(r)$ is strictly increasing on $(0,1)$.
		
		Since $L_1(0)<0$ and $L_1(r)\to+\infty$ as $r\to1^{-}$, the monotonicity of
		$L_1(r)$ implies that $L_1(r)$ has exactly one zero in $(0,1)$. Let $r_f(M)$ be the
		unique root of $L_1(r)$ in $(0,1)$. Then
	    $L_1(r_f(M))=0,$
		which is equivalent to
		\begin{align}\label{Eq-2.6}
			r_f(M)\varphi_0\bigl(r_f(M)\bigr)
			+
			2M
			\sum_{n=2}^{\infty}
			\frac{\varphi_n\bigl(r_f(M)\bigr)}{n(n-1)}
			=
			1+2M(1-2\ln 2).
		\end{align}
		If \(0 \leq r < r_f(M)\), then, by the monotonicity of \(L_1\), we have $L_1(r)<L_1(r_f(M)).$
		Hence, it follows from \eqref{Eq-2.6} that
			\begin{align}\label{Eq-2.7}
			r\varphi_0(r)
			+
			2M
			\sum_{n=2}^{\infty}
			\frac{\varphi_n(r)}{n(n-1)}
			\leq
			1+2M(1-2\ln 2).
		\end{align}
		
		Using Lemma A together with inequalities \eqref{Eq-2.5} and \eqref{Eq-2.7}, for
		\(0<|z|=r\leq r_f(M)\), we deduce that
		\begin{align*}
				B(f,\varphi,r)\leq r\varphi_0(r)+2M\sum_{n=2}^{\infty}\frac{\varphi_n(r)}{n(n-1)}
				\le d\bigl(f(0),\partial f(\mathbb{D})\bigr).
		\end{align*}
		Thus the inequality in the result is established.
		
		In order to complete the proof, we need only show that \(r_f(M)\) is best possible. To achieve this, we consider the function \(f_{M}\) defined by
	\begin{align*}
		f_M(z)=z+2M\sum_{n=2}^{\infty}\frac{ z^n}{n(n-1)}.
	\end{align*}
		At \(z=-r\), it is easy to see that
		\begin{align*}
				|f_M(-r)|
			=\left|-r+2M\sum_{n=2}^{\infty}\frac{(-r)^n}{n(n-1)}\right|
			= r+2M\sum_{n=2}^{\infty}\frac{(-1)^{\,n-1}r^n}{n(n-1)}.
		\end{align*}
		Therefore, the distance is given by
		\begin{align*}
			d\bigl(f_M(0),\partial f_M(\mathbb{D})\bigr)
			&=\liminf_{r\to 1}|f(-r)|\\
			&=1+2M\sum_{n=2}^{\infty}\frac{(-1)^{n-1}}{n(n-1)}\\
			&=1+2M(1-2\ln 2).
		\end{align*}
		
		It is easy to show that \(f_M\in P_{\mathcal{H}}^0(M)\), and for \(f=f_M\), we have
	\begin{align*}
			d\bigl(f_M(0),\partial f_M(\mathbb{D})\bigr)
		=1+2M(1-2\ln 2).
	\end{align*}
	
		For \(|z|=r>r_f(M)\) and \(f=f_{M}\), a straightforward computation using \eqref{Eq-2.6} shows that
		\begin{align*}
			B(f_{M},\varphi,r)
			&=r\varphi_{0}(r)+\sum_{n=2}^{\infty} (|a_n|+|b_n|)\,\varphi_n(r)\\
			&>r_f(M)\varphi_{0}\bigl(r_f(M)\bigr)
			+\sum_{n=2}^{\infty} (|a_n|+|b_n|)\,
			\varphi_n\bigl(r_f(M)\bigr)\\
			&=1+2M(1-2\ln 2)\\
			&=d\bigl(f_{M}(0),\partial f_{M}(\mathbb{D})\bigr).
		\end{align*}
		
		This proves that \(r_f(M)\) is best possible. This completes the proof.
\end{proof}
\section{\bf Bohr inequality and its generalization with new basis for the class $\mathcal{W}_{\mathcal{H}}^{0}(\alpha,\beta)$}
For $\alpha \geq 0$ and $0 \leq \beta < 1$, let
$\mathcal{W}_{\mathcal{H}}^{0}(\alpha,\beta)$ denote the class of harmonic
mappings $f=h+\overline{g}$, which is defined by
\[
\mathcal{W}_{\mathcal{H}}^{0}(\alpha,\beta)
=
\left\{
f=h+\overline{g}\in\mathcal{H}^{0}:
{\rm Re}\bigl(h'(z)+\alpha z h''(z)-\beta\bigr)
>
\left|g'(z)+\alpha z g''(z)\right|,\;z\in\mathbb{D}
\right\}.
\]
We observe that the class $\mathcal{W}_{\mathcal{H}}^{0}(\alpha,\beta)$
generalizes several previously studied classes of harmonic mappings, as
$\mathcal{W}_{\mathcal{H}}^{0}(\alpha,0)\equiv
\mathcal{W}_{\mathcal{H}}^{0}(\alpha)\ \text{(see \cite{Ghosh-IJMMS-2020})},
\mathcal{W}_{\mathcal{H}}^{0}(0,\beta)\equiv
\mathcal{P}_{\mathcal{H}}^{0}(\beta)\ \text{(see \cite{Li-CMB-2017})},$
$ \mathcal{W}_{\mathcal{H}}^{0}(1,0)\equiv
\mathcal{W}_{\mathcal{H}}^{0}\ \text{(see \cite{Prajapat-AEJM-2019})},\mathcal{W}_{\mathcal{H}}^{0}(0,0)\equiv
\mathcal{P}_{\mathcal{H}}^{0}\ \text{(see \cite{Li-MSP-2016})}.$\vspace{1.5mm}

The following coefficient estimates and growth theorems for the class $\mathcal{W}_{\mathcal{H}}^{0}(\alpha,\beta)$ serve as essential lemmas for establishing our main theorems.
\begin{lemC}\cite{RajbalaPrajapat2021}
Let $f=h+\overline{g}\in \mathcal{W}_{\mathcal H}^{0}(\alpha,\beta)$ be of the
form \eqref{Eq-1.2} with $b_1=0$. Then for $n\ge 2$, we have
\begin{enumerate}
	\item[(i)] $\displaystyle |a_n| + |b_n|\leq \frac{2(1-\beta)}
	{n\bigl(1+\alpha(n-1)\bigr)}; $\\[1mm]
	
	\item[(ii)] $\displaystyle ||a_n| - |b_n||\leq \frac{2(1-\beta)}
	{n\bigl(1+\alpha(n-1)\bigr)};$\\[1mm]
	
	\item[(iii)] $\displaystyle |a_n|\leq \frac{2(1-\beta)}
	{n\bigl(1+\alpha(n-1)\bigr)}.$
\end{enumerate}

All these results are sharp for the function
\begin{align*}
	f(z)
	=
	z+\sum_{n=2}^{\infty}
	\frac{2(1-\beta)}
	{n\bigl(1+\alpha(n-1)\bigr)}
	\,z^n .
\end{align*}
\end{lemC}
\begin{lemD}\cite{RajbalaPrajapat2021}
If $f=h+\overline{g}\in \mathcal{W}_{\mathcal H}^{0}(\alpha,\beta)$, then
\begin{align}
	|z|
	+
	2\sum_{n=2}^{\infty}
	\frac{(-1)^{\,n-1}(1-\beta)|z|^{n}}
	{\alpha n^{2}+n(1-\alpha)}
	\leq
	|f(z)|
	\leq
	|z|
	+
	2\sum_{n=2}^{\infty}
	\frac{(1-\beta)|z|^{n}}
	{\alpha n^{2}+n(1-\alpha)}.
\end{align}

Both the inequalities are sharp when
\begin{align*}
	f(z)
	=
	z
	+
	\sum_{n=2}^{\infty}
	\frac{2(1-\beta)}
	{\alpha n^{2}+n(1-\alpha)}
	\,z^{n},
\end{align*}
or its rotations.
\end{lemD}
We next establish the sharp Bohr-type inequality for functions in the 
class $\mathcal{W}_{\mathcal{H}_0}(\alpha,\beta)$. By employing 
Lemmas~C and~D, we determine the sharp Bohr radius 
as follows.
\begin{thm}\label{Thm-2.2}
		Let $f \in \mathcal{W}_{\mathcal H}^{0}(\alpha,\beta)$ be given by \eqref{Eq-1.2}  and let $\{\varphi_n(r)\}$ be a sequence of
	nonnegative increasing functions differentiable in $(0,1)$ such that the
	series $\sum_{n=0}^{\infty}\varphi_n(r)
	$ converges locally uniformly with respect to $r\in[0,1)$, and
	\begin{align}\label{Eq-2.9}
	\sum_{n=2}^{\infty}\frac{\varphi_n(0)}{n\bigl(1+\alpha(n-1))}
	<\frac{1}{2(1-\beta)}
	+
	\sum_{n=2}^{\infty}
	\frac{(-1)^{\,n-1}}
	{\alpha n^{2}+n(1-\alpha)}.
	\end{align}
	Then the inequality \eqref{Eqnn-2.2}
	holds for all $|z|=r\leq r_f(\alpha,\beta)$, where $r_f(\alpha,\beta)$ is the unique root
	of the equation
	\begin{align*}
		r\varphi_0(r)
		+
		2(1-\beta)\sum_{n=2}^{\infty}\frac{\varphi_n(r)}{n\bigl(1+\alpha(n-1))}
    	=1
		+
		2\sum_{n=2}^{\infty}
		\frac{(-1)^{\,n-1}(1-\beta)}
		{\alpha n^{2}+n(1-\alpha)}
	\end{align*}
	in $(0,1)$. The radius $r_f(\alpha,\beta)$ is best possible.
\end{thm}
By setting $\varphi_n(r)=r^n$ for all $n\geq 0$, Theorem \ref{Thm-2.2} directly implies the following corollary. Notably, this specific choice of weight functions reduces our general framework to the classical setting, thereby recovering a previously established theorem as a special case. Theorem~\ref{Thm-2.2} may therefore be viewed as a genuine extension of this known result.
\begin{cor}\cite{RajbalaPrajapat2021}
		Let $f \in \mathcal{W}_{\mathcal H}^{0}(\alpha,\beta)$ be given by \eqref{Eq-1.2}. Then the inequality 
	\begin{align*}
		r
		+
		\sum_{n=2}^{\infty}
		\bigl(|a_n|+|b_n|\bigr)r^n	\leq d\bigl(f(0),\partial f(\mathbb{D})\bigr)
	\end{align*} 
	holds for \(|z|=r\le r_1(\alpha,\beta)\), where \(r_1(\alpha,\beta)\) is the unique root of
	\begin{align*}
		r
		+
		2(1-\beta)\sum_{n=2}^{\infty}\frac{r^n}{n\bigl(1+\alpha(n-1))}
		=1
		+
		2\sum_{n=2}^{\infty}
		\frac{(-1)^{\,n-1}(1-\beta)}
		{\alpha n^{2}+n(1-\alpha)}.
	\end{align*}
	in \((0,1)\). The radius \(r_1(\alpha,\beta)\) is best possible.
\end{cor}
\begin{cor}
	Let $f \in \mathcal{W}_{\mathcal H}^{0}(\alpha,\beta)$ be given by \eqref{Eq-1.2}. Let $\varphi_0(r)=1\;
	\text{and}\;
	\varphi_n(r)=(n+1)r^n,\quad n\geq 1.$ Then the inequality 
	\begin{align*}
		r
		+
		\sum_{n=2}^{\infty}
		\bigl(|a_n|+|b_n|\bigr)(n+1)r^n	\leq d\bigl(f(0),\partial f(\mathbb{D})\bigr)
	\end{align*} holds for $|z|=r\le r_2(\alpha,\beta)$ , where $r_2(\alpha,\beta)$ is the unique root of the equation
	\begin{align*}
		r
		+
		2(1-\beta)\sum_{n=2}^{\infty}\frac{(n+1)r^n}{n\bigl(1+\alpha(n-1))}
		=1
		+
		2\sum_{n=2}^{\infty}
		\frac{(-1)^{\,n-1}(1-\beta)}
		{\alpha n^{2}+n(1-\alpha)}.
	\end{align*}
	in \((0,1)\). The radius $r_2(\alpha,\beta)$ is best possible.
\end{cor}
The class \(\mathcal{W}_{\mathcal H}^{0}(\alpha,\beta)\) serves as a natural generalization of
\(\mathcal{W}_{\mathcal H}^{0}(\alpha)\). Indeed, choosing \(\beta=0\) in the defining inequality yields $\mathcal{W}_{\mathcal H}^{0}(\alpha,\beta)
=
\mathcal{W}_{\mathcal H}^{0}(\alpha).$
Therefore, as an immediate consequence of our general setting, we recover the following theorem due to Allu and Halder \cite{Allu-Halder-2021}. 

\begin{cor}\cite[Theorem 2.6]{Allu-Halder-2021}
Let $f \in \mathcal{W}_{\mathcal H}^{0}(\alpha)$ be given by \eqref{Eq-1.2}. Then the inequality 
\begin{align*}
	r
	+
	\sum_{n=2}^{\infty}
	\bigl(|a_n|+|b_n|\bigr)r^n	\leq d\bigl(f(0),\partial f(\mathbb{D})\bigr)
\end{align*} 
holds for \(|z|=r\le r_1(\alpha)\), where \(r_1(\alpha)\) is the unique root of
\begin{align*}
	r
	+
	2\sum_{n=2}^{\infty}\frac{r^n}{n\bigl(1+\alpha(n-1))}
	=1
	+
	2\sum_{n=2}^{\infty}
	\frac{(-1)^{\,n-1}}
	{\alpha n^{2}+n(1-\alpha)}.
\end{align*}
in \((0,1)\). The radius \(r_1(\alpha)\) is best possible.	
\end{cor}
\begin{proof}[\bf Proof of the Theorem \ref{Thm-2.2}]
	Let $f\in \mathcal{W}_{\mathcal H}^{0}(\alpha,\beta)$. Then, in view of Lemma D, we have
\begin{align}\label{Eq-2.10}
	|f(z)|
	\geq
	|z|
 +
2\sum_{n=2}^{\infty}
\frac{(-1)^{\,n-1}(1-\beta)|z|^{n}}
{\alpha n^{2}+n(1-\alpha)},
	\qquad |z|<1.
\end{align}

By taking $\liminf$ as $|z|\to 1$ on both sides of \eqref{Eq-2.10}, we obtain
\begin{align}\label{Eqn-2.11}
	\liminf_{|z|\to 1}|f(z)|
	\geq
	1
+
2\sum_{n=2}^{\infty}
\frac{(-1)^{\,n-1}(1-\beta)}
{\alpha n^{2}+n(1-\alpha)}.
\end{align}

Since $f(0)=0$, from  \eqref{Eq-2.4} and \eqref{Eqn-2.11}, we obtain
\begin{align}\label{Eq-2.12}
	d\bigl(f(0),\partial f(\mathbb{D})\bigr)
	\geq
	1
	+
	2\sum_{n=2}^{\infty}
	\frac{(-1)^{\,n-1}(1-\beta)}
	{\alpha n^{2}+n(1-\alpha)}.
\end{align}
Let $	L_2:[0,1)\to\mathbb{R}$
be defined by
\begin{align*}
	L_2(r)
	=
	r\varphi_0(r)
	+
	2(1-\beta)\sum_{n=2}^{\infty}\frac{\varphi_n(r)}{n\bigl(1+\alpha(n-1))}
-1
-
2\sum_{n=2}^{\infty}
\frac{(-1)^{\,n-1}(1-\beta)}
{\alpha n^{2}+n(1-\alpha)}.
\end{align*}
In view of \eqref{Eq-2.9}, it is easy to see that
\begin{align*}
	L_2(0)
	=		2(1-\beta)\sum_{n=2}^{\infty}\frac{\varphi_n(0)}{n\bigl(1+\alpha(n-1))}
	-1
	-
	2(1-\beta)\sum_{n=2}^{\infty}
	\frac{(-1)^{\,n-1}}
	{\alpha n^{2}+n(1-\alpha)}<0.
\end{align*}
 On the other hand, since $	L_2(r)\to +\infty
\quad \text{as } r\to 1^{-},$
and
\begin{align*}
	L_2'(r)
	=
	r\varphi_0'(r)+\varphi_0(r)
	+
	2(1-\beta)\sum_{n=2}^{\infty}\frac{\varphi^{\prime}_n(r)}{n\bigl(1+\alpha(n-1))}
	>0
\end{align*}
for $r\in(0,1)$, it follows that $L_2(r)$ is strictly increasing on $(0,1)$.

Since $L_2(0)<0$ and $L_2(r)\to+\infty$ as $r\to1^{-}$, the monotonicity of
$L_2(r)$ implies that $L_2(r)$ has exactly one zero in $(0,1)$. Let $r_f(\alpha, \beta)$ be the
unique root of $L_2(r)$ in $(0,1)$. Then
$L_2(r_f(\alpha,\beta))=0,$
which is equivalent to
\begin{align}\label{Eq-2.13}
r_f(\alpha,\beta)\varphi_0(r_f(\alpha,\beta))
+
2(1-\beta)\sum_{n=2}^{\infty}\frac{\varphi_n(r_f(\alpha,\beta))}{n\bigl(1+\alpha(n-1))}
=1
+
2\sum_{n=2}^{\infty}
\frac{(-1)^{\,n-1}(1-\beta)}
{\alpha n^{2}+n(1-\alpha)}.
\end{align}
If \(0 \leq r < r_f(\alpha,\beta)\), then, by the monotonicity of \(L_2\), we have $L_2(r)<L_2(r_f(\alpha,\beta)).$
Hence, it follows from \eqref{Eq-2.13} that
\begin{align}\label{Eq-2.14}
	r\varphi_0(r_f(\alpha,\beta))
+
2(1-\beta)\sum_{n=2}^{\infty}\frac{\varphi_n(r_f(\alpha,\beta))}{n\bigl(1+\alpha(n-1))}
\leq1
+
2\sum_{n=2}^{\infty}
\frac{(-1)^{\,n-1}(1-\beta)}
{\alpha n^{2}+n(1-\alpha)}
\end{align}
Using Lemma C together with inequalities \eqref{Eq-2.12} and \eqref{Eq-2.14}, for
\(0<|z|=r\leq r_f(M)\), we deduce that
\begin{align*}
	B(f,\varphi,r)\leq r\varphi_0(r)
	+
	2(1-\beta)\sum_{n=2}^{\infty}\frac{\varphi_n(r)}{n\bigl(1+\alpha(n-1))}
	\le d\bigl(f(0),\partial f(\mathbb{D})\bigr)
\end{align*}
Thus the inequality in the result is established.

In order to complete the proof, we need only show that \(r_f(M)\) is best possible. To achieve this, we consider the function \(f_{M}\) defined by
\begin{align*}
f^{\beta}_{\alpha}(z)=z+\sum_{n=2}^{\infty}\frac{2(1-\beta)}{n\bigl(1+\alpha(n-1)\bigr)}\,{z}^{\,n}.
\end{align*}
It is not hard to show that \(f^{\beta}_{\alpha}\in \mathcal{W}_{\mathcal H}^{0}(\alpha,\beta)\).
Further, it can be shown that 
\[
d\bigl(f^{\beta}_{\alpha}(0),\partial f_{\alpha}(\mathbb{D})\bigr)
=1
+
2\sum_{n=2}^{\infty}
\frac{(-1)^{\,n-1}(1-\beta)}
{\alpha n^{2}+n(1-\alpha)}.
\]
For \(|z|=r>r_f(\alpha,\beta)\) and \(f^{\beta}_{\alpha}\), a straightforward computation using
\eqref{Eq-2.13} shows that
\begin{align*}
	B(f^{\beta}_{\alpha},\varphi,r)
	&= r\varphi_{0}(r)
	+\sum_{n=2}^{\infty}\bigl(|a_n|+|b_n|\bigr)\varphi_n(r) \\
	&> r_f(\alpha,\beta)\varphi_{0}\bigl(r_f(\alpha,\beta)\bigr)
	+\sum_{n=2}^{\infty}\bigl(|a_n|+|b_n|\bigr)\varphi_n\bigl(r_f(\alpha,\beta)\bigr) \\
	&=1
	+
	2\sum_{n=2}^{\infty}
	\frac{(-1)^{\,n-1}(1-\beta)}
	{\alpha n^{2}+n(1-\alpha)} \\
	&=d\bigl(f^{\beta}_{\alpha}(0),\partial f^{\beta}_{\alpha}(\mathbb{D})\bigr).
\end{align*}
This proves that \(r_f(\alpha,\beta)\) is best possible. This completes the proof.
\end{proof}
\noindent{\bf Acknowledgment:} The authors would be gratefully acknowledge the anonymous referee for their thorough comments and valuable suggestions, which have greatly improved the quality and clarity of the paper. Research of first author is supported by the Anusandhan National Research Foundation (ANRF), Government of India, under File No: SUR/2022/002244.The research of second author is supported by UGC-JRF (NTA Ref. No. 211610135410), New Delhi, India.
\vspace{1.2mm}

\noindent\textbf{Conflict of interest:} The authors declare that there is no conflict  of interest regarding the publication of this paper.\vspace{1.2mm}

\noindent\textbf{Data availability statement:}  Data sharing not applicable to this article as no datasets were generated or analysed during the current study.
\end{document}